\documentclass{ifacconf}
\newif\ifIsArXiv
\IsArXivtrue

\usepackage{graphicx}
\usepackage{natbib}

\usepackage{amsmath}
\usepackage{amssymb}
\usepackage{tikz}
\usetikzlibrary{patterns,arrows, math}
\usepackage{pgfplots}

\usepackage{textcomp}
\def\BibTeX{{\rm B\kern-.05em{\sc i\kern-.025em b}\kern-.08em
    T\kern-.1667em\lower.7ex\hbox{E}\kern-.125emX}}

\usepackage{subcaption}
\usepackage{comment}
\usepackage{enumerate}
\usepackage{placeins}
\graphicspath{{resources/}}

\usepackage{tikz}
\usetikzlibrary{shapes, arrows}
\tikzstyle{block} = [draw, fill=white, rectangle, minimum height=0em, minimum width=0em]
\tikzstyle{output} = [coordinate]
\tikzstyle{input} = [coordinate]

\newtheorem{algo}[thm]{Algorithm}
\newcommand{\setdef}[2]{\left\{\, #1\, \left|\, \vphantom{#1} #2 \right.\right\}}

\newcommand{\R}{\mathbb{R}}

\newcommand{\N}{\mathbb{N}}

\newcommand{\cU}{\mathcal{U}}

\newcommand{\cD}{\mathcal{D}}
\newcommand{\cF}{\mathcal{F}}
\newcommand{\cG}{\mathcal{G}}

\newcommand{\ve}{\varepsilon}
\newcommand{\eps}{\varepsilon}
\renewcommand{\phi}{\varphi}

\newcommand{\fa}{\ \forall \, }
\newcommand{\ex}{\ \exists \, }
\newcommand{\rbl}{\left (}
\newcommand{\rbr}{\right )}
\newcommand{\sbl}{\left [}
\newcommand{\sbr}{\right ]}

\newcommand{\bl}{\left |}
\newcommand{\br}{\right |}
\newcommand{\nl}{\left\|}
\newcommand{\nr}{\right\|}

\newcommand{\Abs}[1]{\bl #1 \br}
\newcommand{\Norm}[2][ ]{\nl #2 \nr_{#1}}
\newcommand{\SNorm}[1]{\Norm[\infty]{#1}}

\newcommand{\GL}{\text{GL}}

\newcommand{\con}{\mathcal{C}}

\newcommand{\Rp}{\R_{\geq0}}

\DeclareMathOperator*{\loc}{loc}
\DeclareMathOperator*{\rf}{ref}

\begin{document}
\begin{frontmatter}

\title{Towards Funnel MPC for nonlinear systems with relative degree two\thanksref{footnoteinfo}} 

\thanks[footnoteinfo]{D. Dennst\"adt gratefully thanks the Technische Universit\"at Ilmenau and the
    Free State of Thuringia for their financial support as part of the Th\"uringer
Graduiertenf\"orderung.
\ifIsArXiv\\
   This work has been submitted to IFAC for possible publication.
\fi
}

\author[First]{Dario Dennstädt}

\address[First]{Technische Universit\"at Ilmenau,
        Weimarer~Str.~25, 98693 Ilmenau, Germany (dario.dennstaedt@tu-ilmenau.de).}
\begin{abstract}
    Funnel MPC, a novel Model Predictive Control (MPC) scheme, allows guaranteed output tracking of
    smooth reference signals with prescribed error bounds for nonlinear multi-input multi-output
    systems. To this end, the stage cost resembles the high-gain idea of funnel control. However,
    rigorous proofs for initial and recursive feasibility without incorporating additional output
    constraints in the Funnel MPC scheme are only available for systems with relative degree one and
    stable internal dynamics. In this paper, we extend these results to systems with relative degree
    two by incorporating also a term based on the idea of a derivative funnel in the stage cost.
\end{abstract}


\end{frontmatter}
%===============================================================================

\section{Introduction}
Model Predictive Control~(MPC) is a widely-used control technique for linear and nonlinear
systems and has seen various applications, see e.g. \cite{QinBadg03}. Key reasons for its success
are its applicability to multi-input multi-output nonlinear systems and its ability to directly take
control and state constraints into account. To this end, a finite-horizon Optimal Control Problem
(OCP) is solved before the prediction horizon is shifted forward in time and the procedure is
repeated ad infinitum, see e.g. \cite{CoroGrun20} or the textbook \cite{grune2017nonlinear,
rawlings2017model} nicely illustrating the basic concept for discrete-time systems.

\textit{Recursive feasibility} is essential for successfully applying MPC, see e.g.~\cite{EsteWort20}.
This means, solvability of the OCP at a particular time instant has to automatically imply
solvability of the OCP at the successor time instant. 
In order to achieve this, often, suitably designed terminal
conditions (cost and constraints) are incorporated in the OCP to be solved at each time instant, see~\cite{ChenAllg98} or the textbook
~\cite{rawlings2017model} and the references therein.
However, such (artificially introduced) terminal conditions increase the computational burden
of solving the OCP and complicate
the task of finding an initially-feasible solution. % by imposing additional state constraints. 
As a consequence, the domain of the MPC feedback controller might become significantly smaller,
see e.g. \cite{chen2003terminal,gonzalez2009enlarging}.
This technique becomes considerably more involved in the presence
of time-varying state constraints, see e.g. \cite{manrique2014mpc} and references therein.

To overcome these restrictions for a large system class, Funnel MPC (FMPC) was proposed
in~\cite{berger2019learningbased}. This allows output tracking such that the tracking error
evolves in a pre-specified, potentially time-varying performance funnel. A
``funnel-like'' stage cost, which penalizes the tracking error and becomes infinite when approaching
the funnel boundary, is used. By incorporating output constraints in the~OCP and using properties
of the system class in consideration, initial and recursive feasibility are shown
--~without imposing additional terminal conditions  and independent of the length of the prediction
horizon.

The novel stage cost used in FMPC is inspired by  funnel control, a model-free output-error
feedback controller first proposed in~\cite{IlchRyan02b}, see also the recent work by~\cite{BergIlch21}
for a comprehensive literature overview. The funnel controller is an adaptive controller which
allows output tracking within a prescribed performance funnel for a fairly large class of systems
solely invoking structural assumptions, i.e.\ stable internal dynamics, known relative
degree, and a sign-definite high-frequency gain matrix.

It is shown in~\cite{BergDenn21} that such funnel-inspired stage cost automatically ensure initial
and recursive feasibility for a class of nonlinear systems with relative degree one and, in a
certain sense, input-to-state stable internal dynamics.
Since the requirement of a sign-definite gain matrix is omitted, the system class is larger than
the one the original funnel controller is applicable to.
Moreover, adding (artificial) output constraints to the~OPC, as used in the prior work, is superfluous.
In numerical simulations, FMPC shows superior performance compared to both MPC with quadratic stage
cost and funnel control.

Based on these simulations, it was suspected that these results also hold true for systems with
higher relative degree. We show that this is in fact true and that for the scalar case the findings
in~\cite{BergDenn21} can be generalized to systems with relative degree two. However, while previous
results allow for an arbitrary short prediction horizon, for this system class a sufficiently long
horizon --~depending on the funnel~-- is necessary. A further generalization of these results to MIMO
systems with relative degree two can be found in~\cite{Denn22}.

\textbf{Notation:}
$\N$ and $\R$ denote natural and real numbers, resp. $\N_0:=\N\cup\{0\}$ and
$\Rp:=[0,\infty)$. $\Norm{\cdot}$~denotes a norm in $\R^n$. $\GL_n(\R)$ is the group of invertible
$\R^{n\times n}$ matrices. $\con^p(V,\R^n)$ is the linear space of $p$-times continuously
differentiable functions $f:V\to\R^n$, where $V\subset\R^m$ and $p\in\N_0\cup \{\infty\}$. We use the
notation $\con(V,\R^n):=\con^0(V,\R^n)$ to refer to the space of continuous functions.
On an interval $I\subset\R$,  $L^\infty(I,\R^n)$ denotes the space of measurable essentially bounded
functions $f: I\to\R^n$ and $L^\infty_{\text{loc}}(I,\R^n)$ the space of
locally bounded measurable functions. 
Further, $W^{k,\infty}(I,\R^n)$ is the Sobolev space of all $k$-times weakly differentiable functions
$f:I\to\R^n$ such that $f,\dots, f^{(k)}\in L^{\infty}(I,\R^n)$.\\

\section{System class and control objective}
In this section the problem statement is introduced.
We present the considered system class and the control objective and recall some necessary definitions.

\subsection{System class}
We consider control affine multi-input multi-ouptput systems
\begin{equation}\label{eq:Sys}
    \begin{aligned}
        \dot{x}(t)  & = f(x(t)) + g(x(t)) u(t),\quad x(t^0)=x^0,\\
        y(t)        & = h(x(t)),
    \end{aligned}
\end{equation}
with~$t^0\in\Rp$, $x^0\in\R^n$, functions~$f\in \con^2(\R^n,\R^n)$, $g\in \con^2(\R^n,\R^{n\times
m})$, $h \in \con^3(\R^n,\R^m)$ and a control function $u \in L^\infty_{\loc}(\R_{\geq 0}, \R^m)$.
The system~\eqref{eq:Sys}
has a \textit{solution} in the sense of \textit{Carath\'{e}odory}, that is a function $x:[t^0,\omega)\to\R^n$,
$\omega>t^0$, with $x(t^0)=x^0$ which is absolutely continuous and
satisfies the ODE in~\eqref{eq:Sys} for almost all $t\in [t^0,\omega)$.
The \textit{response} associated with $u$ is any maximal solution of~\eqref{eq:Sys} and is denoted
by~$x(\cdot;t^0,x^0,u)$. It is unique since the right-hand side of~\eqref{eq:Sys} is locally Lipschitz
in~$x$.

We recall the notion of relative degree for system~\eqref{eq:Sys}, see e.g.~\cite[Sec. 5.1]{Isid95}.
Assuming that $f,g,h$ are sufficiently smooth, the Lie derivative of~$h$ along~$f$ is defined by
$\rbl L_f h\rbr(x):= h'(x) f(x)$. Lie derivatives of higher order are recursively defined by~$L_f^k
h := L_f (L_{f}^{k-1} h)$, for $k \in \mathbb{N}$, with~$L_f^0 h = h$.
Furthermore, for the matrix-valued function~$g$ we have
\[
    (L_gh)(x) := \sbl (L_{g_1}h)(x), \ldots, (L_{g_m}h)(x) \sbr,
\]
where~$g_i$ denotes the~$i$-th column of~$g$ for $i=1,\ldots, m$.
Then system~\eqref{eq:Sys} is said to have \emph{(global and strict) relative degree}~$r \in \mathbb{N}$, if
\begin{align*}
    \fa k \in \{1,\ldots,r-1\}\fa x \in \R^n:\
        (L_g L_f^{k-1} h)(x)  = 0&  \\
      \text{and}\quad
        (L_g L_f^{r-1} h)(x)  \in \GL_m(\R).&
\end{align*}
If~\eqref{eq:Sys} has relative degree~$r$, then, under the additional assumptions provided
in~\cite[Cor.~5.6]{ByrnIsid91a}, there exists a diffeomorphic coordinate transformation
\begin{align}\label{eq:ExitencePhi}
    \Phi\!:\!\R^n\!\to\R^n,\Phi(x(t))=(y(t), \dot y(t),\ldots, y^{(r-1)}(t),\eta(t))
\end{align}
which puts the system into Byrnes-Isidori form
\begin{subequations}\label{eq:BIF}
    \begin{align}
        y^{(r-1)}(t) &= p\big( y(t), \dot y(t),\ldots, y^{(r-1)}(t),\eta(t)\big)\notag \\
        &\quad + \gamma\big( y(t), \dot y(t),\ldots, y^{(r-1)}(t),\eta(t)\big)\,u(t),
        \label{eq:output_dyn}\\
        \dot \eta(t) &= q\big( y(t), \dot y(t),\ldots, y^{(r-1)}(t),\eta(t)\big),\label{eq:zero_dyn}
    \end{align}
\end{subequations}
where $p\in\con(\R^n,\R^{m})$, $q\in\con(\R^n,\R^{n-rm})$, $\gamma = (L_g L_f^{r-1} h) \circ
\Phi^{-1}\in\con(\R^n,\R^{m\times m})$ and $(y(t^0),\dot y(t^0),\ldots, y^{(r-1)}(t^0),\eta(t^0)) =  \Phi(x^0)$.
Furthermore, we require the following \emph{bounded-input, bounded-state} (BIBS) condition on the
internal dynamics~\eqref{eq:zero_dyn}:
\begin{multline}\label{eq:BIBO-ID}
        \fa c_0 >0  \ex c_1 >0  \fa t^0\ge 0 \fa  \eta^0\in\R^{n-rm} \\
       \fa  \zeta\in L^\infty_{\loc}([t^0,\infty),\R^{rm}):\ \Norm{\eta^0}+
        \SNorm{\zeta}  \leq c_0\\ \implies\ \SNorm{\eta (\cdot;t^0,\eta^0,\zeta)} \leq c_1,\,  
\end{multline}
where $\eta (\cdot;t^0, \eta^0,\zeta):[t^0,\infty)\to\R^{n-rm}$ denotes the unique global solution
of~\eqref{eq:zero_dyn} when $(y,\ldots,y^{(r-1)})$ is substituted by~$\zeta$. The maximal solution
$\eta (\cdot;t^0, \eta^0,\zeta)$ can indeed be extended to a global solution due to the BIBS
condition~\eqref{eq:BIBO-ID}. 

Throughout this paper we will assume that the system~\eqref{eq:Sys} has relative degree~$r=2$ and
that there exists a diffeomorphism~$\Phi$ as in~\eqref{eq:ExitencePhi} which puts the system into
the Byrnes-Isidori form~\eqref{eq:BIF}.

\subsection{Control objective}
The objective is to design a control strategy which allows the output tracking of  given
reference trajectories~$y_{\rf}\in W^{2,\infty}(\Rp,\R^m)$ within pre-specified error bounds. To be
precise, the tracking error $t\mapsto e(t):=y(t)-y_{\rf}(t)$ and its derivative~$\dot{e}(t)$
shall evolve within the prescribed performance funnels
\begin{align*}
    \cF_{\psi_i}:= \setdef{(t,e)\in \Rp\times\R^{m}}{\Norm{e} < \psi_{i}(t)}, \quad i=0,1,
\end{align*}
see also Figure~\ref{Fig:funnel}.
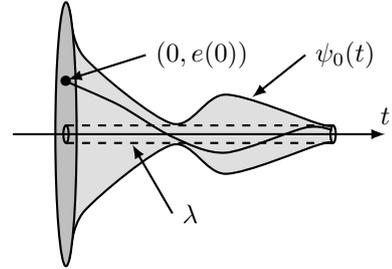
\begin{figure}[h]
\begin{center}
\begin{tikzpicture}[scale=0.35]
\tikzset{>=latex}
  \filldraw[color=gray!25] plot[smooth] coordinates {(0.15,4.7)(0.7,2.9)(4,0.4)(6,1.5)(9.5,0.4)(10,0.333)(10.01,0.331)(10.041,0.3) (10.041,-0.3)(10.01,-0.331)(10,-0.333)(9.5,-0.4)(6,-1.5)(4,-0.4)(0.7,-2.9)(0.15,-4.7)};
  \draw[thick] plot[smooth] coordinates {(0.15,4.7)(0.7,2.9)(4,0.4)(6,1.5)(9.5,0.4)(10,0.333)(10.01,0.331)(10.041,0.3)};
  \draw[thick] plot[smooth] coordinates {(10.041,-0.3)(10.01,-0.331)(10,-0.333)(9.5,-0.4)(6,-1.5)(4,-0.4)(0.7,-2.9)(0.15,-4.7)};
  \draw[thick,fill=lightgray] (0,0) ellipse (0.4 and 5);
  \draw[thick] (0,0) ellipse (0.1 and 0.333);
  \draw[thick,fill=gray!25] (10.041,0) ellipse (0.1 and 0.333);
  \draw[thick] plot[smooth] coordinates {(0,2)(2,1.1)(4,-0.1)(6,-0.7)(9,0.25)(10,0.15)};
  \draw[thick,->] (-2,0)--(12,0) node[right,above]{\normalsize$t$};
  \draw[thick,dashed](0,0.333)--(10,0.333);
  \draw[thick,dashed](0,-0.333)--(10,-0.333);
  \node [black] at (0,2) {\textbullet};
  \draw[->,thick](4,-3)node[right]{\normalsize$\lambda$}--(2.5,-0.4);
  \draw[->,thick](3,3)node[right]{\normalsize$(0,e(0))$}--(0.07,2.07);
  \draw[->,thick](9,3)node[right]{\normalsize$\psi_0(t)$}--(7,1.4);
\end{tikzpicture}
\end{center}
 \vspace*{-2mm}
 \begin{center}
     \caption{Error evolution in a funnel $\mathcal F_{\psi_0}$ with boundary $\psi_0(t)$.}
     \label{Fig:funnel}
 \end{center}
 \end{figure}

\vspace*{-6mm}
These funnels are determined by the choice of the functions $\psi_0$, $\psi_1$ belonging to
\begin{align*}
    \cG^0:=\setdef
        {\psi\in W^{2,\infty}(\Rp,\R)}
        {
            \inf_{t\geq 0}\psi(t) > 0
        }.
\end{align*}
Note that the funnel $\psi_i$ is uniformly bounded away from zero; i.e.\ there exists a boundary
$\lambda>0$ with $\psi_i(t)>\lambda$ for all $t\geq0$. Thus, perfect or asymptotic tracking is not our
control objective. However, $\lambda$ can be be arbitrarily small.
Furthermore, the funnel boundary is not necessarily monotonically decreasing.

If the error~$e$ evolves within the funnel~$\cF_{\psi_0}$ for some $\psi_0\in\cG^0$, then the
derivative~$\dot{e}$ has to satisfy at some point $t\geq0$
\[
    \dot{e}(t)<\dot{\psi}_0(t)\text{ or } \dot{e}(t)>-\dot{\psi}_0(t).
\]
Thus, the derivative funnel must be large enough for the error~$e$ to follow the funnel
boundary~$\psi_0$ and we therefore assume that $\psi=(\psi_0,\psi_1)$ is an element of
\[
    \cG^1\!\!:=\!\setdef
    {\!\!(\psi_0,\psi_1)\!\in\! \cG^0\!\!\times\!\cG^0\!\!}
    {\!\!\!
        \ex \ve>\!0\!\fa t\geq0: \psi_1(t)\!\geq\!\ve-\dot{\psi}_0(t)
    \!\!}.
\]
Typically, the specific application dictates constraints on the tracking error and thus
indicates suitable choices for~$\psi$. 

\section{Funnel MPC} % scheme}

In order to extend the results from~\cite{BergDenn21} to systems of the form~\eqref{eq:Sys} with
relative degree two, we define, for $y_{\rf}\in W^{2,\infty}$, $t\geq0$, and
$\zeta=(\zeta^0,\zeta^1)\in\R^{2m}$,
\[
    e_i(t,\zeta):=\zeta^i-y^{(i)}_{\rf} \quad \text{ for } i =0,1.
\]
We propose for $\psi=(\psi_0,\psi_1)\in\cG^1$ and the design parameter~$\lambda_u\geq0$ the
new~\textit{stage cost function}
\begin{equation}\label{eq:stageCostFunnelMPC}
    \begin{aligned}
        &\ell:\Rp\times\R^{2m}\times\R^{m}\to\R\cup\{\infty\}, \\
       &(t,\zeta,u)\!\! \mapsto\!\!
        \begin{cases}
           \!\sum\limits_{i=0}^1\! \tfrac{1}{1- \Norm{e_i(t,\zeta)}^2/\psi_i(t)^2}\!+\! \lambda_u\!
           \Norm{u}^2\!\!\!, &
           \!\!\!\!\begin{array}{l}
           \Norm{e_i(t,\zeta)} \!\!\neq\! \psi_i(t)\\
           \text{for } i=0,1\end{array}\\
            \!\infty,&
           \!\!\!\!\begin{array}{l}
            \text{else}.
            \end{array}
        \end{cases}
    \end{aligned}
\end{equation}
By setting $\zeta=(y(t),\dot{y}(t))$, the terms $\frac{1}{1- \Norm{e_i(t,\zeta)}^2/\psi_i(t)^2}$
penalize the distance of the tracking error~$e(t)=y(t)-y_{\rf}(t)$ and its derivative~$\dot{e}(t)$
to their respective funnel boundaries~$\psi_i(t)$. The parameter~$\lambda_u$ allows to adjust a
suitable trade off between tracking performance and required control effort. Note that we allow
$\lambda_u=0$. %
The stage cost~$\ell$ is motivated by the construction of the funnel
controller in~\cite{HackHopf13} which also introduces an additional funnel for the derivative in
order to generalize the results from~\cite{IlchRyan02b} to systems with relative degree two.

Based on the stage cost~\eqref{eq:stageCostFunnelMPC}, we may define the Funnel MPC
(FMPC) algorithm as follows.
\begin{algo}[FMPC]\label{Algo:MPCFunnelCost}\ \\
    \textbf{Given:}
    System~\eqref{eq:Sys},  reference signal $y_{\rf}\in
    W^{2,\infty}(\Rp,\R^{m})$, funnel function $\psi=(\psi_0,\psi_1)\in\cG^1$,
    stage cost function~$\ell$ as in~\eqref{eq:stageCostFunnelMPC},
    $M>0$, $t^0\in\Rp$, and $x^0\in\R^n$ with
    \begin{align}\label{eq:DefinitionDt}
        \Phi(x^0)\! \in\! \cD_{t^0}\!:=
        \!\setdef{\!\!(\zeta,\eta)\!\in\!\R^{2m}\!\!\times\!\R^{n-2m}\!\!}
        {\!\!\begin{array}{l}
                \|e_i(t_0,\zeta)\|\!<\psi_i(t_0)\\
        \text{for } i=0,1 \end{array}\!\!},
    \end{align}
    where~$\Phi$ is the diffeomorphism from~\eqref{eq:ExitencePhi}.\\
    \textbf{Set} the time shift $\delta >0$, the prediction horizon $T>\delta$ and initialize the current time
        $\hat t :=t^0$.\\
    \textbf{Steps:}
    \begin{enumerate}[(a)]
    \item\label{agostep:FunnelMPCFirst}
        Obtain a measurement of the state~$x = \Phi^{-1}(y,\dot{y},\eta)$ at time~$\hat t$
        and set $\hat x:=x(\hat t)$.
    \item
        Compute a solution $u^{\star}\in L^\infty([\hat t,\hat t +T],\R^{m})$ of the Optimal
        Control Problem (OCP)
    \begin{equation}\label{eq:OCP}
    \begin{alignedat}{2}
            &\!\mathop
            {\operatorname{minimize}}_{u\in L^{\infty}([\hat t,\hat t+T],\R^{m})}  && \
            \int_{\hat t}^{\hat t+T} \ell\big(t,(y(t), \dot{y}(t)),u(t)\big) {\rm d}t \\
            &\quad \text{subject to}       &\ \ & \eqref{eq:Sys},\ x(\hat t)  = \hat{x},         \\
            &                        && \Norm{u(t)}  \leq M\ \text{ for } t\in [\hat t,\hat t +T]\\
    \end{alignedat}
    \end{equation}
    \item Apply the feedback law
        \begin{equation}\label{eq:FMPC-fb}
            \mu:[\hat t,\hat t+\delta)\times\R^n\to\R^m, \quad \mu(t,\hat x) =u^{\star}(t)
        \end{equation}
        to system~\eqref{eq:Sys}.
        Increase $\hat t$ by $\delta$ and go to Step~\eqref{agostep:FunnelMPCFirst}.

    \end{enumerate}
\end{algo}

In practical application there usually is a limitation $M>0$ on the maximal control that can be applied
to the system~\ref{eq:Sys}. To ensure that the control signal meets this bound, the constraint
$\Norm{u(t)}\leq M$ is added to the OPC~\eqref{eq:OCP} of the FMPC Algorithm~\ref{Algo:MPCFunnelCost}.

\section{Main result}

 Our main results is to show that for scalar systems the Funnel MPC
 Algorithm~\ref{Algo:MPCFunnelCost} is, given a sufficiently long prediction horizon~$T>0$ and large
 enough control constraint~$M>0$, initially and recursively feasible and that it guarantees the
 evolution of the tracking error~$e$ and its derivative~$\dot{e}$ within their respective
 performance funnels $\cF_{\psi_i}$.
\begin{thm}\label{Thm:FunnnelMPC}
    Consider scalar system~\eqref{eq:BIF} with strict relative degree~$r=2$ and $m=1$.
    Assume that there exists a diffeomorphism $\Phi:\R^n\to\R^n$ such that the coordination
    transformation in~\eqref{eq:ExitencePhi} puts the system~\eqref{eq:Sys} in the
    Byrnes-Isidori form~\eqref{eq:BIF} satisfying~\eqref{eq:BIBO-ID}.
    Let $\psi=(\psi_0,\psi_1)\in\cG^1$, $y_{\rf}\in
    W^{2,\infty}(\Rp,\R)$, $t^0\in\Rp$, $\delta>0$, and $B\subset \cD_{t^0}$ a compact set.
    Then there exist $T>\delta$ and $M>0$ such that the  FMPC
    Algorithm~\ref{Algo:MPCFunnelCost} is initially and
    recursively feasible for every $x^0\in B$, i.e.\ at time $\hat t = t^0$ and at each
    successor time $\hat t\in t^0+\delta\N$ the OCP~\eqref{eq:OCP}
    has a solution.
    In particular, the closed-loop system consisting of~\eqref{eq:Sys} and the FMPC
    feedback~\eqref{eq:FMPC-fb} has a (not necessarily unique) global solution
    $x:[t^0,\infty)\to\R^n$ and the corresponding input is given by
    \[
        u_{\rm FMPC}(t) = \mu(t,x(\hat t)),\quad t\in [\hat t,\hat t+\delta),\ \hat t\in
        t^0+\delta\N_0.
    \]
    Furthermore, each global solution~$x$ with corresponding input $u_{\rm FMPC}$ satisfies:
    \begin{enumerate}[(i)]
        \item\label{th:item:BoundedInput}
$\fa t\ge t^0:\ \Abs{u_{\rm FMPC}(t)}\leq M$.
        \item\label{th:item:ErrorInFunnel} $\fa t\ge t^0:\ \Abs{e^{(i)}(t)}<\psi_i(t)$ for $i=0,1$; in particular
            the error $e=y-y_{\rf}$ evolves within the funnel $\cF_{\psi_0}$ and $\dot{e}$ within $\cF_{\psi_1}$.
    \end{enumerate}
\end{thm}
\begin{proof}
    We provide a sketch of the proof. For more details we refer to~\cite{Denn22}.
    For $T>0$, $M>0$, $\hat{t}\geq t^0$, $\hat{x}$ with
    $\Phi(\hat{x})\in\cD_{\hat{t}}$ as in~\eqref{eq:DefinitionDt},
    and the interval $I_{\hat{t}}^T:=[\hat{t},\hat{t}+T]$, we denote
    by $\cU^{M}_{T}(\hat{t},\hat{x})$ the set
    \begin{align}\label{eq:SetControls}
        \setdef
        {\!\!u\!\in\! L^\infty(I_{\hat{t}}^T,\R)\!\!}
        { \!\!
            \begin{array}{l}
                    \SNorm{u}\!\!\!<M,\ x(t;\hat{t},\hat{x},u)\text{ satisfies
                    \eqref{eq:Sys} and}\\
                    \Phi(x(t;\hat{t},\hat{x},u))\in\cD_t\text{ for all } t\in I_{\hat{t}}^T\
            \end{array}
        \!\!\!}.
    \end{align}
    This is the set of all $L^\infty$-controls $u$ bounded by $M>0$ which, if applied to
    system~\eqref{eq:Sys}, guarantee that the error~$e(t)=y(t)-y_{\rf}(t)$ and its derivative
    $\dot{e}(t)$ evolve within their respective funnels.
    A straightforward adaption of Theorem~4.3 and Theorem~4.5 from~\cite{BergDenn21} to the
    current setting yields: if the set~$\cU^{M}_{T}(\hat{t},\hat{x})$ is non-empty,
    then the OCP~\eqref{eq:OCP} has a solution $u^{\star}\in\cU^{M}_{T}(\hat{t},\hat{x})$.
    Therefore, if $u^{\star}$ is applied to system~\eqref{eq:Sys}, then $\|e^{(i)}(t)\|<\psi_i(t)$ for
    all $t\in I^T_{\hat{t}}$ and $i=0,1$. In particular
    $\Phi(x(\hat{t}+\delta;\hat{t},\hat{x},u^\star))\in\cD_{\hat{t}+\delta}$.
    Thus, it is sufficient to show that there exist~$T>0$ and $M>0$ such that
    $\Phi(\hat{x})\in\cD_{\hat{t}}$ implies the non-emptiness of set~$\cU^{M}_{T}(\hat{t},\hat{x})$
    at time $\hat t = t^0$ and at each successor time $\hat t\in t^0+\delta\N$ during the
    Funnel MPC Algorithm~\ref{Algo:MPCFunnelCost}.

    Since the set $B$ is compact, there exists $\tilde{\ve}>0$ such that for all initial values $x^0$
    of system~\eqref{eq:Sys} with $\Phi(x^0)\in B$ the initial tracking error $e^0=y(t^0)-y_{\rf}(t^0)$
    satisfies $\Abs{e^0}<\psi_0(t^0)-\tilde{\ve}$ and $\tilde{\ve}<\inf_{t\geq 0}\psi_0(t)$.
    Using the Byrnes-Isidori form~\eqref{eq:BIF} and (BIBS) condition~\eqref{eq:BIBO-ID},
    one can show that there exists $M>0$ such that for every initial value $x^0$ with $\Phi(x^0)\in
    B$ there exists a control $\bar{u}\in L^\infty([t^0,\infty),\R)$  bounded by $M$ for which the following holds.
    If $\bar{u}$ is applied to~\eqref{eq:Sys}, then there exists $\bar{t}>t^0$ with
    $\Abs{e(t)}<\psi_0(t)$ for all $t\in[t^0,\bar{t}]$,
    $\psi_0(\bar{t})-\tilde{\ve}<\Abs{e(\bar{t})}< \psi_0(\bar{t})$, and
    $\Abs{\dot{e}(t)-\dot{\psi}_0(t)}=0$ for all $t\geq\bar{t}$.
    Thus, the distance of tracking error~$e$ to the upper (lower) funnel boundary remains constant from
    time $\bar{t}$ onwards. 
    Hence, $\Abs{e(t)}<\psi_0(t)$ for all $t\geq t^0$.
    Since there exists~$\ve>0$ with $\psi_1(t)\geq \ve-\dot{\psi}_0(t)$ for all $t\geq0$, the derivative $\dot{e}$
    also stays within the funnel boundary~$\psi_1$ from time $\bar{t}$ onwards. Choosing $M>0$ large
    enough, this can also be achieved up to $\bar{t}$. Thus, $\bar{u}\in\cU^{M}_{T}(t^0,x^0)$ for
    all $T>0$. The bound $M$ depends on $\psi_0$, $\psi_1$, $\tilde{\ve}$, the set $B$,  and the functions~$p$,
    $\gamma$, $q$ from~\eqref{eq:BIF}. An explicit construction of $M>0$ and $\bar{u}$ can be found
    in~\cite{Denn22}.

    Similar, if the tracking error $e$ satisfies $\Abs{e(\hat{t})}<\psi_0(\hat{t})-\tilde{\ve}$ at
    a time $\hat t\in t^0+\delta\N$, then $\cU^{M}_{T}(\hat{t},\hat{x})\neq\emptyset$. Otherwise,
    $\cU^{M}_{T-\delta}(\hat{t},\hat{x})$ is non empty, since the solution~$u^\star$ of the
    OCP~\eqref{eq:OCP} from the previous time step is an element of $\cU^{M}_{T-\delta}(\hat{t},\hat{x})$.
    Choosing~$\tilde{T}:=T-\delta$ large enough, it is possible to proof  that there exists a control
    $\hat{u}\in\cU^{M}_{T-\delta}(\hat{t},\hat{x})$ for which the following holds. 
    If $\hat{u}$ is applied to the system~\eqref{eq:Sys}, there exists a $\tilde{t}>\hat{t}$ with either
    $\Abs{e(\tilde{t})}<\psi_0(\tilde{t})-\tilde{\ve}$ or $\Abs{\dot{e}(\tilde{t})-\dot{\psi}_0(t)}=0$.
    Details on the construction of~$\hat{u}$ depending on a large enough horizon $T-\delta$ can be
    found in~\cite{Denn22}.

    In the first case, it follows with the previous reasoning that
    $\cU^{M}_{T}(\tilde{t},x(\tilde{t};\hat{t},\hat{x},\hat{u}))$ is not empty. This implies
    $\cU^{M}_{T}(\hat{t},\hat{x})\neq\emptyset$ since for all
    $u\in\cU^{M}_{T}(\tilde{t},x(\tilde{t};\hat{t},\hat{x},\hat{u}))$, the control $\check{u}$
    defined by $\check{u}\vert_{[\hat{t},\tilde{t}]} =\hat{u}$ and
    $\check{u}\vert_{[\tilde{t},\hat{t}+T]} =u$ is an element of $\cU^{M}_{T}(\hat{t},\hat{x})$.

    In the latter case, one can construct $\tilde{u}$ bounded by $M$ for which the following holds.
    $\tilde{u}\vert_{[\hat{t},\tilde{t}]} = \tilde{u}\vert_{[\hat{t},\tilde{t}]}$ and
    if $\tilde{u}$ is applied to the system~\eqref{eq:Sys},
    $\Abs{\dot{e}(\tilde{t})-\dot{\psi}_0(t)}=0$ for all $t\geq \tilde{t}$. 
    The application of $\tilde{u}$ guarantees that tracking error~$e$ remain 
    remains with constant distance within an~$\tilde{\eps}$--margin to the upper (lower) funnel boundary
    from time $\tilde{t}$ onwards.
    Thus, $\Abs{e(t)}<\psi_0(t)$ for all $t\geq\hat{t}$. Since there exists~$\ve>0$ with
    $\psi_1(t)\geq \ve-\dot{\psi}_0(t)$ for all $t\geq0$, the derivative $\dot{e}$ also stays within
    the funnel boundary~$\psi_1$. Hence, $\Abs{\dot{e}(t)}<\psi_1(t)$ for all $t\geq \hat{t}$.
    Therefore, $\tilde{u}\in\cU^{M}_{T}(\hat{t},\hat{x})$. The horizon $T$ depends on the funnel
    boundaries $\psi_0$, $\psi_1$, and the bound $M>0$. This completes the proof.
\end{proof}

\subsubsection{Remark} 
Note that proving the recursive feasibility of the OCP is not trivial. Similar to the proof of the
recursive feasibility of the FMPC Algorithm in~\cite{BergDenn21}, the main challenge is to guarantee
that the set of controls $\cU^{M}_{T}(\hat{t},\hat{x})$ as in~\eqref{eq:SetControls} is non-empty at
each time step~$\hat{t}$ of the Funnel MPC Algorithm~\ref{Algo:MPCFunnelCost}. While for systems with
relative degree one it is possible to find $M>0$ such that $\cU^{M}_{T}(\hat{t},\hat{x})$ is
non-empty at time step~$\hat{t}$ independent of the horizon $T>0$, for systems with relative degree
two this seems not to be possible. 
An explicit construction of $M$ and $T$ such that $\cU^{M}_{T}(\hat{t},\hat{x})\neq\emptyset$ is
guaranteed can be found in~\cite{Denn22} together with a generalization of
Theorem~\ref{Thm:FunnnelMPC} to MIMO systems, i.e\ $m>1$.

Note further that while the proof of Theorem~\ref{Thm:FunnnelMPC} makes extensive use of the
diffeomorphism~$\Phi$ as in~\eqref{eq:ExitencePhi} and the Byrnes-Isidori form~\eqref{eq:BIF}, their
computation is not necessary for the application of the FMPC Algorithm~\ref{Algo:MPCFunnelCost}.
Condition~\eqref{eq:DefinitionDt} requires the initial tracking
error~$e(t^0)$ and its derivative~$\dot{e}(t^0)$ to be within their respective funnel boundaries.
This can easily be verified by different means.

\section{Simulation}
    To demonstrate the application of the FMPC Algorithm~\ref{Algo:MPCFunnelCost}, we consider the
    example of a mass-spring system mounted on a car from~\cite{SeifBlaj13}. The mass $m_2$
    moves on a ramp inclined by the angle $\vartheta\in[0,\frac{\pi}{2})$ and mounted on a car with mass
    $m_1$ by a spring-damper system, see Figure~\ref{Mass.on.car}.
    \begin{figure}[htp]
        \begin{center}
        \includegraphics[trim=2cm 4cm 5cm 15cm,clip=true,width=6.5cm]{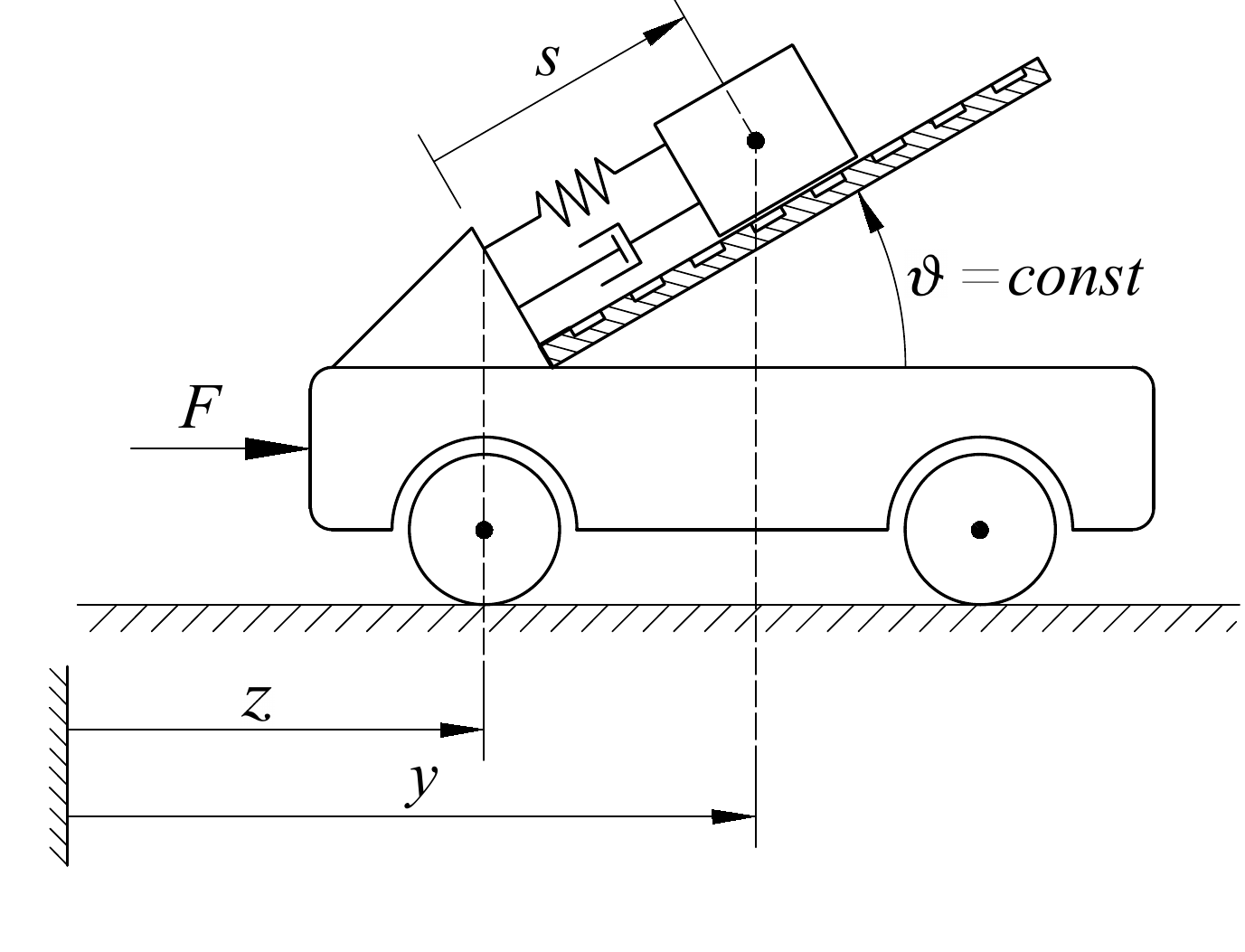}
        \end{center}
        \vspace*{-3mm}
        \caption{Mass-on-car system.}
        \label{Mass.on.car}
    \end{figure}
    It is possible to control the force~$F=u$ acting on the car.
     The motion of the system is described by the equations
    \begin{align}\label{eq:ExampleMassOnCarSystem}
    \begin{bmatrix}
        m_1 + m_2& m_2\cos(\vartheta)\\
        m_2 \cos(\vartheta) & m_2
    \end{bmatrix}\!\!
    \begin{pmatrix}
        \ddot{z}(t)\\
        \ddot{s}(t)
    \end{pmatrix}\!\!
    +\!\!
    \begin{pmatrix}
        0\\
        k s(t) +d\dot{s}(t)
    \end{pmatrix}
    \!\!=\!\!
    \begin{pmatrix}
        u(t)\\
        0
    \end{pmatrix}\!,
\end{align}
where $z(t)$ is the horizontal position of the car and $s(t)$ the relative position of the mass on
the ramp at time $t$.
The physical constants $k>0$ and $d>0$ are the coefficients of the spring and
damper, resp. The horizontal position of the mass on the ramp is the output $y$ of the system, i.e.
\begin{align}\label{eq:ExampleMassOnCarSystemOutput}
    y(t)=z(t)+s(t)\cos(\vartheta).
\end{align}
We choose the parameters $m_1=4$, $m_2=1$, $k=2$, $d=1$, $\vartheta=\tfrac{\pi}{4}$ and initial values
$z(0)=s(0)=\dot{z}(0)=\dot{s}(0) = 0$ for the simulation.
The objective is tracking of the reference signal $y_{\rf}:t\mapsto \cos(t)$ so that 
the error $t\mapsto e(t):=y(t)-y_{\rf}(t)$ satisfies $\Abs{e(t)}\leq\psi_0(t)$ and 
$\Abs{\dot{e}(t)}\leq\psi_1(t)$ for all $t\geq0$ with
 \[
     \psi_0(t)=3\e^{-2t}+0.1,\qquad \psi_1(t)=6e^{-t}+0.1.
 \]
One can easily verify that ~$\psi=(\psi_0,\psi_1)\in\cG^1$ and that the initial errors lie within
their respective funnel boundaries.
As shown in~\cite[Section 3.1]{BergIlch21}, the system~\eqref{eq:ExampleMassOnCarSystem} with
output~\eqref{eq:ExampleMassOnCarSystemOutput} has relative degree $r=2$ for the given parameters.
We compare the FMPC Algorithm~\ref{Algo:MPCFunnelCost} with stage cost~\eqref{eq:stageCostFunnelMPC}
to the FMPC scheme from~\cite{BergDenn21} which uses the stage cost function
\begin{equation}\label{eq:stageCostFunnelMPC1Funnel}
    \begin{aligned}
        &\tilde{\ell}:\Rp\times\R^{2m}\times\R^{m}\to\R\cup\{\infty\}, \\
       & (t,\zeta,u)\! \mapsto\!\!
        \begin{cases}
           \tfrac{1}{1- \Norm{e_0(t,\zeta)}^2/\psi_0(t)^2} + \lambda_u \Norm{u}^2\!\!,& \!\!\!\!\Norm{e_0(t,\zeta)} \!\neq\! \psi_0(t) \\
            \infty,&\!\!\text{else.}
        \end{cases}
    \end{aligned}
\end{equation}
The function~$\tilde{\ell}$ penalizes the distance of the tracking error~$e(t)=y(t)-y_{\rf}(t)$
to the funnel boundary~$\psi_0$ but, contrary to the stage cost
function~\eqref{eq:stageCostFunnelMPC}, not the distance of the
derivative~$\dot{e}(t)=\dot{y}(t)-\dot{y}_{\rf}(t)$ to the boundary~$\psi_1$.
In both cases we choose for the FMPC Algorithm~\ref{Algo:MPCFunnelCost} the prediction horizon
$T=0.6$ and time shift $\delta=0.04$.
Due to discretisation, only step functions with constant step length $0.04$ are considered for the
OCP~\eqref{eq:OCP}.
We further choose for both stage cost functions the parameter ${\lambda_u=5\cdot10^{-3}}$ and allow
a maximal control value of $M = 30$.
All simulations are performed on the time interval $[0,7]$ with the MATLAB routine \texttt{ode45}
and are depicted in Figure~\ref{Fig:SimulationFunnelMPC}.
Figure~\ref{Fig:SimulationOutputError} shows the tracking error of the two different FMPC schemes
evolving within the funnel boundaries given by~$\psi_0$, while
Figure~\ref{Fig:SimulationOutputErrorDot} displays the derivative of the error within the boundaries
given by~$\psi_1$ The respective control signals generated by the controllers is displayed in
Figure~\ref{Fig:SimulationControlInput}.

\begin{figure}[ht] \centering
    \begin{subfigure}[b]{0.45\textwidth}
     \centering
        \includegraphics[width=8.2cm]{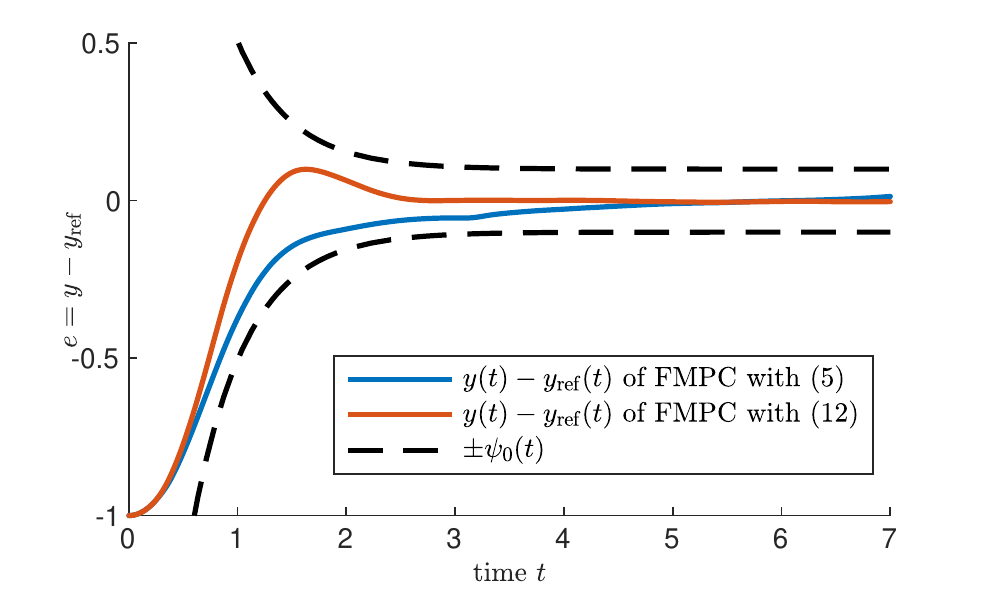}
        \caption{Funnel given by~$\psi_0$ and tracking error~$e$}
        \label{Fig:SimulationOutputError}
    \end{subfigure}
\end{figure}
\begin{figure}[ht]\ContinuedFloat
    \centering
    \begin{subfigure}[b]{0.45\textwidth}
        \centering
        \includegraphics[width=8.2cm]{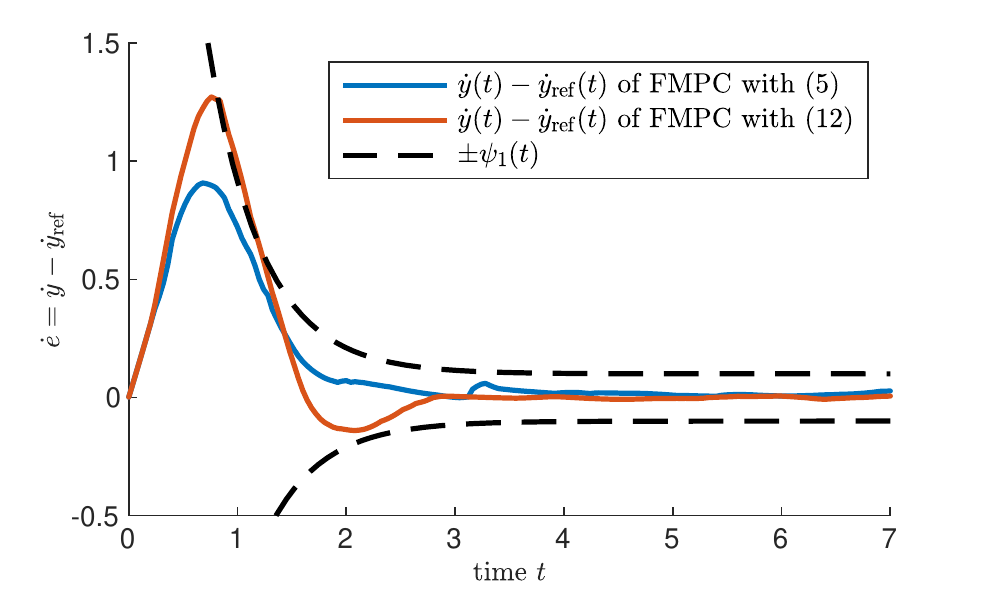}
        \caption{Funnel given by~$\psi_1$ and tracking error derivative~$\dot{e}$}
    \label{Fig:SimulationOutputErrorDot}
    \end{subfigure}
\end{figure}
\begin{figure}[ht]\ContinuedFloat 
    \centering
    \begin{subfigure}[b]{0.45\textwidth}
        \centering
        \includegraphics[width=8.2cm]{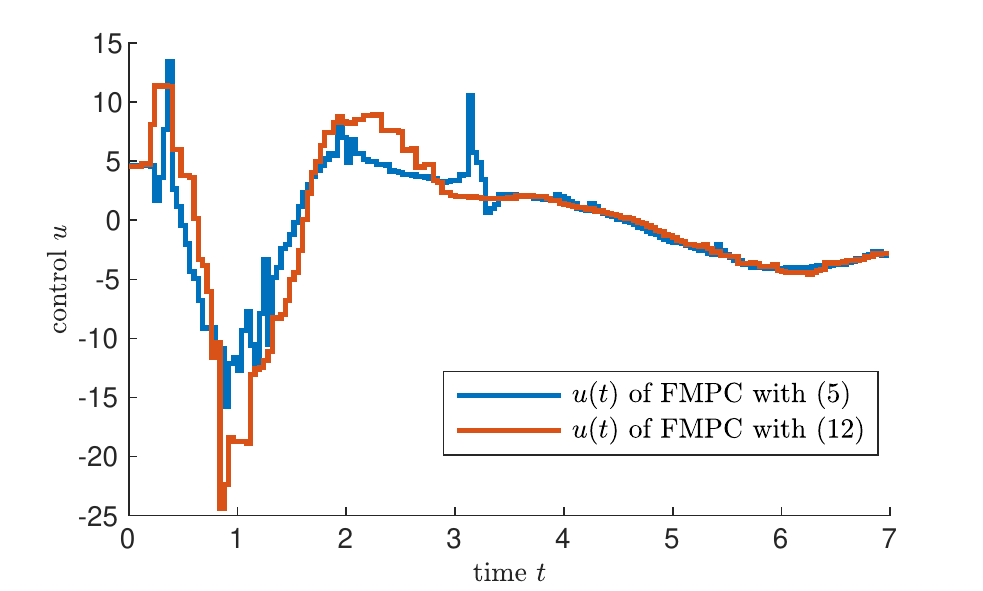}
        \caption{Control input}
        \label{Fig:SimulationControlInput}
    \end{subfigure}
    \caption{Simulation of system~\eqref{eq:ExampleMassOnCarSystem} with
    output~\eqref{eq:ExampleMassOnCarSystemOutput} under FMPC Algorithm~\ref{Algo:MPCFunnelCost} and FMPC from~\cite{BergDenn21}}
    \label{Fig:SimulationFunnelMPC}
\end{figure}
It is evident that both control schemes achieve the tracking of the reference signal within the
performance boundaries given by~$\psi_0$.
While the FMPC Algorithm~\ref{Algo:MPCFunnelCost} with stage cost function \eqref{eq:stageCostFunnelMPC} also
ensures that the derivative of the tracking error evolves within funnel given by~$\psi_1$,
FMPC scheme with stage cost function~$\tilde{\ell}$ as in~\eqref{eq:stageCostFunnelMPC1Funnel} fails
to do that and thus does not achieve the overall control objective.
This is not surprising since the function~$\tilde{\ell}$ does not penalize the distance of error's derivative
to the funnel boundary.
Moreover, the FMPC Algorithm~\ref{Algo:MPCFunnelCost} with stage cost function
\eqref{eq:stageCostFunnelMPC} exhibits a smaller range of employed control values as the FMPC scheme
from~\cite{BergDenn21}.

\section{Conclusion}

    In this note we outline a conceptual framework to extend the FMPC scheme proposed
    in~\cite{BergDenn21}, which solves the problem of tracking a reference signal within a
    prescribed performance funnel, to systems with relative degree two.
    By exploiting concepts from funnel control
    and using a ``funnel-like'' stage cost, feasibility is achieved without the need for
    additional terminal or explicit output constraints while also being restricted to (a priori) bounded
    control values. In particular, additional output constraints in the OCP of FMPC as considered
    in~\cite{berger2019learningbased} are not required to infer the feasibility results. However,
    contrary to previous results the prediction horizon has to be sufficiently long in order to
    guarantee recursive feasibility of the Funnel MPC algorithm.

    Extending these results to systems with arbitrary relative degree~$r>2$ is subject of future work.

\begin{ack}
    I thank Thomas Berger (Universität Paderborn), Achim Ilchmann (Tu Ilmenau), and Karl Worthmann
    (Tu Ilmenau) for many helpful discussions, suggestions, and corrections.
\end{ack}

\bibliography{references}
\appendix
\end{document}